\documentclass[11pt,reqno]{amsart}
\usepackage[utf8]{inputenc}
\usepackage{amsthm,amsmath,amsxtra,latexsym,amscd,amssymb,exscale,mathrsfs,mathtools,physics,xparse,xcolor}
\usepackage[top=3.5cm, bottom=3.5cm, left=3.5cm, right=3.5cm]{geometry}
\usepackage{scalerel,geometry,boxedminipage,enumitem}
\usepackage[abbrev]{amsrefs}
\usepackage[colorlinks=true,linkcolor=blue]{hyperref}
\usepackage[depth=2]{bookmark}
\newtheorem{theorem}{Theorem}[section]
\newtheorem{lemma}[theorem]{Lemma}
\newtheorem{proposition}[theorem]{Proposition}
\newtheorem{corollary}[theorem]{Corollary}

\theoremstyle{definition}
\newtheorem{definition}[theorem]{Definition}
\theoremstyle{definition}
\newtheorem{remark}[theorem]{Remark}
\theoremstyle{definition}

\theoremstyle{definition}

\numberwithin{equation}{section}

\let\div\undefined
\DeclareMathOperator{\div}{div}

\let\vol\undefined
\newcommand{\vol}{\textnormal{vol}}

\newcommand{\sff}{\mathrm{I\;\!\!I}}
\let\tr\undefined
\DeclareMathOperator{\tr}{tr}
\let\sgn\undefined
\DeclareMathOperator{\sgn}{sgn}

\let\uball\undefined
\newcommand{\uball}{\mathbb{B}}
\let\usphere\undefined
\newcommand{\usphere}{\mathbb{S}}
\let\Real\undefined
\newcommand{\Real}{\mathbb{R}}

\let\D\undefined
\newcommand{\D}{\textnormal{D}}

\let\S\undefined
\newcommand{\S}{\mathrm{S}}
\let\W\undefined
\newcommand{\W}{\mathcal{W}}

\title[Monotonicity formula for anisotropic minimal hypersurfaces]{A monotonicity formula for anisotropic minimal hypersurfaces}
\author{Doanh Pham}
\address{Beijing International Center for Mathematical Research, Peking University, Beijing, China}
\email{doanhpham@pku.edu.cn}

\calclayout

\begin{document}
\begin{abstract}
Under a sign assumption on the Minkowski norm, we prove a monotonicity formula for anisotropic minimal hypersurfaces in Euclidean space.
\end{abstract}
	\maketitle

\textsc{Important note:} After the initial version of this note was posted, we learned that in dimension $n+1 \geq 3$, a Minkowski norm $F$ satisfying Condition S \eqref{eq: sign condition for F} necessarily takes the form \eqref{eq: norm-typed function}, as proved in \cite[p.~440]{Allard_1974_area_integrand}.

\section{Introduction}
Minimal hypersurfaces arise as critical points of the Euclidean area functional and are characterized by the condition that their mean curvature vanishes identically. One of the fundamental tools in the analysis of minimal hypersurfaces of the Euclidean space is the monotonicity formula, which provides a quantitative measure of how area concentrates near a point. Specifically, let $M^n \subset \Real^{n+1}$ be a smooth minimal hypersurface. The classical monotonicity formula (see, e.g., \cite{Simon_Lectures on geometric measure theory}) asserts that the normalized area ratio
\[ r \longmapsto \frac{|M \cap \uball^{n+1}_r|}{r^n} \]
is monotone increasing as long as $\partial M \cap \uball^{n+1}_r = \emptyset$, where $\uball^{n+1}_r$ denotes the open ball of radius $r$ centered at the origin in $\Real^{n+1}$. This formula plays an important role in regularity theory and the study of singularities.
	
In many variational problems (see e.g. \cite{Chodosh-Li_ForumPi2023, DePhilippis-DeRosa_CPAM2024_anisotropic_min-max, DeRosa-Kolasinski-Santilli_ARMA2020_Uniqueness_critical_anisotropic_isoperimetric, Du-Yang_MathAnn2024_flatness_anisotropic, Figalli-Maggi_ARMA2011_liquid_drops_crystals, Mooney-Yang_Invent2024_Bernstein} and the references therein), the isotropic area functional is replaced by an anisotropic energy
\[\mathcal{A}_F(M) = \int_M F(\nu) \, d\vol_M\]
where the Minkowski norm $F : \Real^{n+1} \backslash \{0\} \to (0, \infty)$ is a smooth positively $1$-homogeneous function satisfying the elliptic condition that $\D^2 F^2$ is positive definite on $\Real^{n+1} \backslash \{0\}$ (which is equivalent to \eqref{eq: elliptic condition of F}). Critical points of this energy functional are called anisotropic minimal (or $F$-minimal) hypersurfaces and are characterized by the vanishing of their anisotropic mean curvature. While these hypersurfaces share many similarities with their isotropic counterparts, the loss of symmetry gives rise to substantial geometric and analytic challenges. In particular, no monotonicity formula is known for arbitrary Minkowski norms. In this note, we prove a monotonicity formula for $F$-minimal hypersurfaces provided that $F$ satisfies a suitable sign condition. Let $F^\circ$ be the dual Minkowski norm of $F$. We will impose the following condition:
\begin{flalign}\label{eq: sign condition for F}
\textnormal{\textbf{Condition S:}} \quad \sgn \langle \D F(u), \D F^{\circ}(v) \rangle = \sgn \langle u, v \rangle \quad \text{for all}\;\; u, v \in \Real^{n+1} \backslash\{0\}.
\end{flalign}
Let $\W$ be the Wulff shape of $F$ which bounds a convex domain $\Omega$ of $\Real^{n+1}$. The main result of this note is the following statement:
\begin{theorem} \label{thm: monotonicity for anisotropic minimal}
	Let $M$ be an oriented smooth hypersurface of $\Real^{n+1}$ with a chosen unit normal vector field $\nu$. Suppose that $M$ is $F$-minimal and satisfies $\partial M \subset \bar{r}\W$ for some $\bar{r} > 0$. Then, for every $0 < s < r < \bar{r}$, we have
	$$\frac{1}{r^n} \int_{M \cap (r\Omega)} F(\nu) - \frac{1}{s^n} \int_{M \cap (s\Omega)} F(\nu) = \int_{M \cap (r\Omega) \backslash (s\Omega)} \frac{\langle \D F^{\circ}(x), \D F(\nu) \rangle \langle x, \nu \rangle}{F^{\circ}(x)^{n+1}} \; d\vol_M(x).$$
	In particular, if $F$ satisfies Condition S \eqref{eq: sign condition for F}, then the normalized anisotropic energy
	$$r \longmapsto \frac{1}{r^n} \int_{M \cap (r\Omega)} F(\nu)$$
	is monotone increasing, and is constant if and only if $M$ is a hyperspace of $\Real^{n+1}$.
\end{theorem}

As a consequence, we obtain the sharp lower bound for area of anisotropic minimal hypersurfaces of the Wulff shape passing through the origin.
\begin{corollary} \label{cor: lower bound for anisotropic minimal in Wulff}
Let $M$ be a $F$-minimal hypersurface of $\Omega$ which contains the origin and satisfies $\partial M \subset \mathcal{W}$. Suppose that $F$ satisfies Condition S \eqref{eq: sign condition for F}. Then
\[\int_M F(\nu) \geq F(\nu(0)) \left|\Omega \cap T_0M \right|.\]
	The equality holds if and only if $M$ is a hyperspace of $\Real^{n+1}$.
\end{corollary}

A simple example of Minkowski norms satisfying Condition S (\ref{eq: sign condition for F}) is a function of the form
\begin{equation}\label{eq: norm-typed function}
	F(u) = \sqrt{\langle Au, u \rangle}
\end{equation}
for any positive definite symmetric matrix $A$. Specifically, in this case, $F^{\circ}(v) = \sqrt{\langle A^{-1}v, v \rangle}$ and the following stronger condition holds
	\begin{equation}\label{eq: condition of Ferone-Kawohl}
		\langle \D F(u), \D F^{\circ}(v) \rangle = \frac{\langle u, v \rangle}{F(u)F^{\circ}(v)} \quad \text{for}\;\; u, v \ne 0.
	\end{equation}
To the author's knowledge, condition \eqref{eq: condition of Ferone-Kawohl} first appeared in \cite{Ferone-Kawohl_PAMS2009} where it was shown to be both necessary and sufficient for the mean value property of anisotropic harmonic functions. These functions are solutions of
$$\Delta_F u \coloneqq \div (\D Q^{\circ} (\D u)) = 0,$$
where $Q^\circ \coloneqq \frac{1}{2} F^2$. Condition S \eqref{eq: sign condition for F} was later introduced in \cite{Cozzi-Farina-Valdinoci_Adv2016_Monotonicity_anisotropic_PDE}, where the authors established monotonicity formulas for solutions to a class of anisotropic equations. Moreover, they observed that any positive, $1$-homogeneous convex function satisfying both the ellipticity condition and \eqref{eq: condition of Ferone-Kawohl} must necessarily be of the form \eqref{eq: norm-typed function}. They also provided a characterization of even Minkowski norms in $\Real^2$ that satisfy the ellipticity condition together with Condition S \eqref{eq: sign condition for F}.

As a byproduct of the preliminary computations underlying the proof of the main result, we obtain a Minkowski-type formula for affine $k$-curvatures. Details are given in the appendix.
\medskip

\noindent \textit{Conventions and notations}. In this note, we use $\D$ to denote the standard derivatives on the Euclidean space $\Real^{n+1}$. Let $M$ be a hypersurface in $\Real^{n+1}$. The Euclidean connection on $M$ is denoted by $\D^M$, and the space of smooth tangent vector fields on $M$ is denoted by $\mathfrak{X}(M)$. A choice of unit normal vector field is denoted by $\nu$ and the second fundamental form with respect to $\nu$ is denoted by $\sff_\nu$. We also use $\sff_\nu$ to denote the shape operator $X \in TM \mapsto -\D_X \nu$. The mean curvature of $M$ is given by $H_{\nu} = \tr(\sff_{\nu})$. For $x \in M$ and $\zeta \in \Real^{n+1}$, we denote by $\zeta^\top$ the tangential part of $\zeta$ on $T_x M$ and by $\zeta^\perp$ its orthogonal complement. We denote by $\eta$ the outer unit conormal vector field on $\partial M$.

For a linear map $L : \Real^{n+1} \to \Real^{n+1}$ and $x \in M$, we denote by $[L]_M$ the linear map given by $\eval{L}_{T_xM}$ followed by the projection onto $T_xM$. Specifically, $[L]_M: T_xM \to T_xM$ is the linear map satisfying
\[[L]_M (X) = L(X)^\top \quad \text{for every}\;\; X \in T_xM.\]
	
\section{Preliminaries}

\subsection{Minkowski norms and Wulff shapes}
Throughout this note, let $\Omega \subset \Real^{n+1}$ be a uniformly convex bounded domain containing the origin with smooth boundary $\mathcal{W}$. The support function $F$ of $\Omega$ is given by
$$F(u) \coloneqq \sup_{x \in \Omega} \langle u, x \rangle = \max_{x \in \mathcal{W}}\, \langle u, x \rangle, \quad u \in \Real^{n+1}.$$
In the literature, $F$ is called a \textit{Minkowski norm} and $\mathcal{W}$ is called the \textit{Wulff shape} of $F$. It is well-known that $F$ is a positively $1$-homogeneous convex function, smooth on $\Real^{n+1} \backslash \{0\}$, which satisfies
\begin{equation}\label{eq: elliptic condition of F}
	\D^2_{\usphere^n} F(u) + F(u)I_n > 0 \quad \text{on}\;\; T_u\usphere^n \;\; \text{for every}\;\; u \in \usphere^n.
\end{equation}
By Euler homogeneous theorem, we have
\begin{equation} \label{eq: Euler homo for F}
	F(u) = \langle \D F(u), u \rangle \quad \text{for all} \;\; u \in \Real^{n+1}.
\end{equation}
Differentiating \eqref{eq: Euler homo for F} gives
\begin{align}\label{eq: [D^2F(u)](u) = 0}
	[\D^2F(u)](u) = \vec{0} \quad \text{for all} \;\; u \in \Real^{n+1}.
\end{align}

The \textit{polar set} of $\Omega$ is the convex bounded domain defined by
$$\Omega^{\circ} \coloneqq \{u: \langle u, x \rangle \leq 1 \;\; \text{for all}\;\; x \in \Omega\} = \{u: F(u) \leq 1\}.$$
The \textit{dual Minkowski norm} of $F$ is the support function $F^{\circ}$ of $\Omega^{\circ}$. Then $F^{\circ}$ is smooth on $\Real^{n+1}\backslash\{0\}$ and satisfies
\[F^{\circ}(v) = \sup_{u \ne 0} \frac{\langle u, v \rangle}{F(u)}, \quad v \in \Real^{n+1}.\]
In addition, $F^\circ = 1$ on $\mathcal{W}$.

\subsection{Affine geometry}
	
In fact, we will prove a slightly more general result of the main theorem which is stated in affine geometry setting. Basic notions in this subsection can be found in \cite{Nomizu-Sasaki_Book_Affine_geometry}. Let $M$ be a hypersurface of $\Real^{n+1}$ and suppose that $\xi$ is a smooth transversal vector field on $M$. We consider the following decompositions in $\Real^{n+1} = T_x M \oplus \textnormal{Span}\{\xi_x\}$:
\begin{align}
\D_X Y &= \nabla_X Y + h(X, Y)\xi \quad \text{for} \;\; X, Y \in \mathfrak{X}(M); && (\text{Gauss equation}) \label{eq: Gauss equation}\\
\D_X \xi &= -\S(X) + \tau(X)\xi \quad\;\; \text{for} \;\; X \in \mathfrak{X}(M). && (\text{Weingarten equation}) \label{eq: Weingarten equation}
\end{align}
Then, Gauss equation \eqref{eq: Gauss equation} defines a torsion-free connection $\nabla$ which is called the \textit{induced affine connection}, and a symmetric $2$-form $h$ which is called the \textit{affine (second) fundamental form} on $M$. Moreover, Weingarten equation \eqref{eq: Weingarten equation} defines a $(1,1)$ tensor $\S$ which is called the \textit{affine shape operator}, and a $1$-form $\tau$ which is called the \textit{transversal connection form} on $M$. With all these notions, $(M, \xi)$ is an \textit{affine hypersurface} of $\Real^{n+1}$. The \textit{affine mean curvature} of $(M, \xi)$ is $H_{\xi} \coloneqq \tr(\S)$.
\begin{definition}
A transversal vector field $\xi$ on $M$ is said to be \textit{equiaffine} if $\D_X \xi \in T_x M$ for every $X \in T_x M$ at each $x \in M$ (i.e., the transversal connection form vanishes identically). In this case, $(M, \xi)$ is called an \textit{equiaffine hypersurface} of $\Real^{n+1}$.
\end{definition}

\subsection{Anisotropic geometry}
Let $M$ be an oriented hypersurface of $\Real^{n+1}$ equipped with a chosen unit normal vector field $\nu$. The \textit{anisotropic normal} on $M$ is the vector field $\nu_F \coloneqq \D F(\nu)$. Since $\langle \nu_F, \nu \rangle = F(\nu) > 0$, it follows that $\nu_F$ is a transversal vector field on $M$. Moreover, by \eqref{eq: [D^2F(u)](u) = 0}, we have
\[\D_X \nu_F = \D_X (F(\nu)) = \langle [\D^2 F(\nu)](\nu), X \rangle = 0 \quad \text{for all} \;\; X \in \mathfrak{X}(M).\]
Thus $(M, \nu_F)$ is an equiaffine hypersurface of $\Real^{n+1}$. With $\xi = \nu_F$ in \eqref{eq: Weingarten equation}, the \textit{anisotropic mean curvature} is $H_{\nu_F} = \tr(X \mapsto -\D_X \nu_F) = -\div_M \nu_F$ on $M$. 

\begin{definition}
We say that $M$ is \textit{$F$-minimal} if its anisotropic mean curvature vanishes identically.
\end{definition}

More details on anisotropic geometry can be found in, for example, \cite{He-Li_Sinica2008_Integral_Minkowski, Jia-Wang-Xia-Zhang_ARMA2023_Alexandrov, Xia_IUMJ2013_AnisotropicMinkowski} and the references therein.
	
\section{Monotonicity formula}

In this section, we prove the main result. To do this, we first define an appropriate notion of the tangential part of a vector field and prove some useful formulas.

\begin{definition}
Let $\xi$ be a transversal vector field on $M$. For $V \in \Real^{n+1}$, the\textit{ affine tangential part} of $V$ at $x \in M$ is defined by
\[V^{\top_\xi} \coloneqq (V \wedge \xi) \nu = \langle \xi, \nu \rangle V - \langle V, \nu \rangle \xi.\]
In other words, $V$ is decomposed by
\[V = \frac{V^{\top_\xi} }{\langle \xi, \nu \rangle} + \frac{\langle V, \nu \rangle}{\langle \xi, \nu \rangle} \xi.\]
\end{definition}

\begin{remark}
With this notation, the Gauss equation \eqref{eq: Gauss equation} means
\begin{equation} \label{eq: Gauss equation rewritten}
\nabla_X Y = \frac{(\D_X Y)^{\top_\xi}}{\langle \xi, \nu \rangle} \quad \text{and} \quad h(X, Y) = \frac{\sff_\nu(X, Y)}{\langle \xi, \nu \rangle} \quad \text{for all} \;\; X, Y \in \mathfrak{X}(M).
\end{equation}
\end{remark}  
Given a vector field $X$ on an equiaffine hypersurface $(M, \xi)$, we derive some formulas for $X^{\top_\xi}$.

\begin{lemma} \label{lem: D^M X^top_xi}
	Let $(M, \xi)$ be an equiaffine hypersurface. Suppose that $X$ is a smooth vector field on $M$, then
	\[\D^M X^{\top_\xi} = \langle \xi, \nu \rangle [\D X]_M - X^\top \otimes \sff_\nu(\xi^\top) - \xi^\top \otimes \D^M \langle X, \nu \rangle + \langle X, \nu \rangle \S.\]
This means that for a local orthonormal frame $\{e_i\}$ on $M$ we have
\begin{align*}
&\left\langle \D_{e_j} (X^{\top_\xi}), e_i \right\rangle \\
&= \langle \xi, \nu \rangle \langle \D_{e_j}X, e_i \rangle - \langle X, e_i \rangle \sff_{\nu}(\xi^\top, e_j) - \langle \xi, e_i \rangle \langle \D^M\langle X, \nu \rangle, e_j \rangle + \langle X, \nu \rangle \langle \S(e_j), e_i \rangle.
\end{align*}
In particular, we have
\begin{equation} \label{eq: div_M X^{top_xi}}
\div_M X^{\top_\xi} = \langle \xi, \nu \rangle \div_M X + \langle X, \nu \rangle H_{\xi} - \left\langle \sff_\nu(X^\top) + \D^M \langle X, \nu \rangle, \xi \right\rangle.
\end{equation}
\end{lemma}

\begin{proof}
Since $\xi$ is equiaffine on $M$, we have
\[\D_{e_j} \langle \xi, \nu \rangle = \langle \xi, \D_{e_j} \nu \rangle = - \langle \sff_{\nu}(\xi^\top), e_j \rangle. \]
Therefore $\D^M \langle \xi, \nu \rangle = - \sff_{\nu}(\xi^\top)$. Using this, we get
\[\D_{e_j}(X^{\top_\xi}) = \langle \xi, \nu \rangle \D_{e_j} X - \sff_{\nu}(\xi^\top, e_j) X - \langle \D^M \langle X, \nu \rangle, e_j \rangle \xi - \langle X, \nu \rangle \D_{e_j} \xi.\]
The rest of the proof follows easily from here.
\end{proof}

\begin{lemma} \label{lem: div_M x^top_xi}
	Suppose that $X$ is a smooth vector field on an equiaffine hypersurface $(M, \xi)$. If $\D_E X$ is tangent for every $E \in TM$, then
	\[\div_M X^{\top_\xi} = \langle \xi, \nu \rangle \div_M X + \langle X, \nu \rangle H_{\xi}.\]
	In particular, we have
	\begin{enumerate}[leftmargin=\parindent+0.5cm,label=\textnormal{(\roman*)}]
		\item $\div_M b^{\top_\xi} = \langle b, \nu \rangle H_{\xi}$ where $b$ is a constant vector.
		\item $\div_M x^{\top_\xi} = n \langle \xi, \nu \rangle + \langle x, \nu \rangle H_{\xi}$ where $x$ is the position vector.
	\end{enumerate}
\end{lemma}

\begin{proof}
Since $\D_E X$ is tangent for every $E \in TM$ (i.e. $X$ satisfies the equiaffine condition), similarly to the proof of Lemma \ref{lem: D^M X^top_xi}, we have $\D^M \langle X, \nu \rangle = -\sff_\nu(X^\top)$. The conclusion follows easily from \eqref{eq: div_M X^{top_xi}}.
\end{proof}

\begin{lemma} \label{lem: div_M fX^top_xi}
	For every smooth function $f$ and vector field $X$ on an equiaffine hypersurface $(M, \xi)$, we have
	\begin{align*}
		\div_M(f X^{\top_\xi}) &= f \div_M X^{\top_\xi} + \langle \D^M f \wedge \nu, \, X \wedge \xi \rangle \\
		&= f \div_M X^{\top_\xi} + \langle \xi, \nu \rangle \langle \D^M f, X \rangle - \langle X, \nu \rangle \langle \D^M f, \xi \rangle.
	\end{align*}
\end{lemma}

\begin{proof}
	We have
	\begin{align*}
		\div_M(f X^{\top_\xi}) &= f \div_M X^{\top_\xi} + \langle \D^M f, X^{\top_\xi} \rangle\\
		&= f \div_M X^{\top_\xi} + \langle \D^M f, \, (X \wedge \xi)(\nu) \rangle \\
		&= f \div_M X^{\top_\xi} + \langle \D^M f \wedge \nu, \, X \wedge \xi \rangle. \qedhere
	\end{align*}
\end{proof}

A monotonicity formula in affine geometry setting is given in the following statement.

\begin{proposition} \label{prop: monotonicity formula for equiaffine}
	Let $(M, \xi)$ be a compact equiaffine hypersurface of $\Real^{n+1}$ with vanishing affine mean curvature and smooth boundary $\partial M$. Suppose that $\varphi \geq 0$ is a positively $1$-homogeneous function on $\Real^{n+1}$, smooth on $\{x: \varphi(x) \neq 0\}$, whose sublevel sets $B^\varphi_r \coloneqq \{x \in \Real^{n+1}: \varphi(x) < r\}$, $r > 0$, have smooth boundaries. If $\partial M \subset \partial B^\varphi_{\bar{r}}$ for some $\bar{r} > 0$, then for every $0 < s < r < \bar{r}$, we have
	\begin{equation*}
		\frac{1}{r^n} \int_{M \cap B^\varphi_r} \langle \xi, \nu \rangle - \frac{1}{s^n} \int_{M \cap B^\varphi_s} \langle \xi, \nu \rangle = \int_{M \cap B^\varphi_r \backslash B^\varphi_s} \frac{\langle x, \nu \rangle \langle \D \varphi(x), \xi \rangle}{\varphi(x)^{n+1}}.
	\end{equation*}
\end{proposition}

\begin{proof}
The main idea, which is suggested by a well-known proof of the classical monotonicity formula for minimal hypersurfaces (see \cite{Brendle_Toulouse2023_MinimalSurvey}), is to consider the following tangent vector field
\[V(x) = \frac{x^{\top_\xi}}{n \varphi(x)^n}, \quad x \in M \backslash \{\varphi = 0\}.\]
Using Lemmas \ref{lem: div_M x^top_xi}, \ref{lem: div_M fX^top_xi}, and the assumption $H_\xi = 0$ on $M$, we compute
	\begin{align*}
		\div_M V(x) &= \frac{\langle \xi, \nu \rangle}{\varphi(x)^n} - \frac{\langle \xi, \nu \rangle \langle \D^M \varphi(x), x \rangle}{\varphi(x)^{n+1}} + \frac{\langle x, \nu \rangle \langle \D^M \varphi(x), \xi \rangle}{\varphi(x)^{n+1}} \\
		&= \frac{\langle \xi, \nu \rangle \langle \D \varphi(x), \nu \rangle \langle x, \nu \rangle}{\varphi(x)^{n+1}} + \frac{\langle x, \nu \rangle \langle \D^M \varphi(x), \xi \rangle}{\varphi(x)^{n+1}} \qquad \text{since } \varphi(x) = \langle \D \varphi(x), x \rangle \\
		&= \frac{\langle x, \nu \rangle \langle \D \varphi(x), \xi \rangle}{\varphi(x)^{n+1}}.
	\end{align*}
	 By divergence theorem and Lemma \ref{lem: div_M x^top_xi} with $H_\xi = 0$ again, it follows that
	\begin{align*}
		\frac{1}{r^n} \int_{M \cap B^\varphi_r} \langle \xi, \nu \rangle - \frac{1}{s^n} \int_{M \cap B^\varphi_s} \langle \xi, \nu \rangle &= \frac{1}{nr^n} \int_{M \cap B^\varphi_r} \div_M x^{\top_\xi} - \frac{1}{ns^n} \int_{M \cap B^\varphi_s} \div_M x^{\top_\xi} \\
		&= \frac{1}{nr^n} \int_{M \cap \partial B^\varphi_r} \langle x^{\top_\xi}, \eta \rangle -  \frac{1}{ns^n} \int_{M \cap \partial B^\varphi_s} \langle x^{\top_\xi}, \eta \rangle \\
		&= \int_{M \cap B^\varphi_r \backslash B^\varphi_s} \div_M V \\
		&= \int_{M \cap B^\varphi_r \backslash B^\varphi_s} \frac{\langle x, \nu \rangle \langle \D \varphi(x), \xi \rangle}{\varphi(x)^{n+1}}. \qedhere
	\end{align*}
\end{proof}

\begin{proof}[Proof of Theorem \ref{thm: monotonicity for anisotropic minimal}]
The main result is a special case of Proposition \ref{prop: monotonicity formula for equiaffine} when $\xi = \nu_F$ and $\varphi = F^\circ$.
\end{proof}

\begin{proof}[Proof of Corollary \ref{cor: lower bound for anisotropic minimal in Wulff}]
This follows from Theorem \ref{thm: monotonicity for anisotropic minimal} by noting that
\begin{align*}
	\lim\limits_{r \to 0} \,\frac{1}{r^n} \int_{M \cap (r\Omega)} F(\nu) = F(\nu(0)) \left| \Omega \cap T_0M \right|
	\end{align*}
since $0 \in M$.
\end{proof}

\appendix

\section{Minkowski-type formula for equiaffine hypersurfaces}

In this appendix we prove a Minkowski-type formula for affine $k$-curvatures, generalizing the anisotropic version in \cite{He-Li_Sinica2008_Integral_Minkowski}. When the affine fundamental form is nondegenerate, the formula is classical (see, e.g., \cite{Nomizu-Sasaki_Book_Affine_geometry}).

\subsection{Elementary symmetric functions and Newton tensors}
Given a square matrix $A = (a_{ij})$ of size $n$ (not necessarily symmetric), let $\lambda_1, \dots, \lambda_n$ denote its (possibly complex) eigenvalues. The elementary symmetric functions are defined by
$$\sigma_0(A) = 1 \quad \text{and} \quad \sigma_k(A) = \sum_{i_1 < \dots < i_k} \lambda_{i_1} \dots \lambda_{i_k} \quad \text{for } 1 \leq k \leq n.$$
Let $E_k(A)$ denote the sum of the principle minors of size $k$ of $A$. It is well-known (see e.g. \cite[Theorem 1.2.16]{Horn-Johnson_Matrix}) that
\begin{equation} \label{eq: sigma_k = sum of minors}
	\sigma_k(A) = E_k(A) = \frac{1}{k!} \sum \delta_{i_1, \dots, i_k}^{j_1, \dots, j_k} a_{i_1j_1} \dots a_{i_kj_k}.
\end{equation}
Here, $\delta_{i_1, \dots, i_k}^{j_1, \dots, j_k}$ denotes the generalized Kronecker symbol which takes value $1$ (resp. $-1$) if $i_1, \dots, i_k$ are distinct and $(i_1, \dots, i_k)$ is an even (resp. odd) permutation of $(j_1, \dots, j_k)$, and takes value $0$ in any remaining case.

For $0 \leq k \leq n$, the Newton tensor $T_k(A)$ is the matrix given by (see e.g. \cite{He-Li_Sinica2008_Integral_Minkowski, DellaPietra-Gavitone-Xia_Adv2021_Symmetry_mixed_volume, Reilly_MichMJ1973_curvatures_graph, Reilly_Duke1976_geometry_nonparametric} with possibly different conventions)
\begin{equation*}
	T_k(A) = \sum_{i = 0}^{k} (-1)^i \sigma_{k-i}(A) A^i = \sigma_k(A) I_n - \sigma_{k-1}(A) A + \dots + (-1)^k A^k,
\end{equation*}
which has a useful recursive formula
\begin{equation}\label{eq: T_k recursive}
	T_{k}(A) = \sigma_{k}(A) I_n - T_{k-1}(A) \cdot A.
\end{equation}
The entries of $T_k(A)$ are computed by (see e.g. \cite{He-Li_Sinica2008_Integral_Minkowski, DellaPietra-Gavitone-Xia_Adv2021_Symmetry_mixed_volume})
\begin{equation} \label{eq: T_k entries}
	[T_{k}(A)]_{ji} = \frac{1}{k!} \sum \delta_{i_1, \dots, i_k, i}^{j_1, \dots, j_k, j} a_{i_1j_1} \dots a_{i_kj_k}.
\end{equation}
In view of (\ref{eq: sigma_k = sum of minors}), this yields (note the reversed order of $i,j$)
\begin{equation}\label{eq: T_{k-1} as gradient of sigma_k}
	[T_{k-1}(A)]_{ji} = \frac{\partial \sigma_{k}(A)}{\partial a_{ij}}, \quad 1 \leq k \leq n.
\end{equation}
Furthermore, since $\sigma_{k}$ is $k$-homogeneous, by (\ref{eq: T_{k-1} as gradient of sigma_k}) and Euler homogeneous theorem, this implies that
\begin{equation}\label{eq: tr [T_k(A)]A}
	\sigma_{k}(A) = \frac{1}{k} [T_{k-1}(A)]_{ji} a_{ij} = \frac{1}{k} \tr \left(T_{k-1}(A) \cdot A \right).
\end{equation}
Taking trace of (\ref{eq: T_k recursive}) and using (\ref{eq: tr [T_k(A)]A}) gives
\begin{equation}\label{eq: tr T_k}
	\tr T_{k}(A) = (n-k) \sigma_{k}(A).
\end{equation}

\subsection{Minkowski-type formula}

We always assume that $(M, \xi)$ is an equiaffine hypersurface of $\Real^{n+1}$ with the affine fundamental form $h$ and the affine shape operator $\S$. 

\begin{lemma} \label{lem: self-adjoint lemma}
For every real polynomial $P$, the tensor $\sff_\nu P(\S)$ is self-adjoint.
\end{lemma}

\begin{proof}
The Ricci formula for equiaffine hypersurfaces (see {\cite[Theorem 2.4 in Chapter II]{Nomizu-Sasaki_Book_Affine_geometry}}) asserts that
\[h(X, \S(Y)) = h(\S(X), Y) \quad \text{for all} \;\; X, Y \in \mathfrak{X}(M).\]
It follows from this and \eqref{eq: Gauss equation rewritten} that the tensor $\sff_\nu \S$ is self-adjoint. In other words,
\[\sff_\nu \S = \S^T \sff_\nu.\]
Thus, for every $m \in \mathbb{N}$, we have
\[\sff_\nu \S^m = \S^T \sff_\nu \S^{m-1} = \dots = (\S^T)^m \sff_\nu = (\sff_\nu \S^m)^T.\]
This shows that $\sff_\nu P(\S)$ is self-adjoint for every real polynomial $P$.
\end{proof}

Next, we prove a Codazzi property for $\S$ with respect to the Euclidean connection on $M$.

\begin{lemma} \label{lem: Codazzi property for affine shape operator}
The affine shape operator satisfies the following property:
\[[\D^M_X \S](Y) = [\D^M_Y \S](X) \quad \text{for all} \;\; X, Y \in \mathfrak{X}(M).\]
\end{lemma}

\begin{proof}
For $X, Y \in \mathfrak{X}(M)$, by the facts that $\S(X) = -\D_X \xi$ and that $\xi$ is equiaffine, we have
\begin{align*}
[\D^M_X \S](Y) - [\D^M_Y \S](X) &= \D^M_X (\S(Y)) - \D^M_Y (\S(X)) - \S(\D^M_X Y - \D^M_Y X)\\
&= (\D_Y \D_X \xi - \D_X \D_Y \xi)^\top + \D_{[X, Y]} \xi\\
&= -  \D_{[X, Y]} \xi + \D_{[X, Y]} \xi\\
&= 0. \qedhere
\end{align*}
\end{proof}

Using Lemma \ref{lem: Codazzi property for affine shape operator}, a standard argument based on \eqref{eq: T_k entries} (see, for example, \cite{Reilly_MichMJ1973_curvatures_graph, He-Li_Sinica2008_Integral_Minkowski}) shows that the Newton tensors associated with the affine shape operator are divergence-free. We record this fact below.

\begin{lemma} \label{lem: T_k is divergence-free}
For each $k = 0, \dots, n-1$, the Newton tensor $T_k(\S)$ of the affine shape operator satisfies the following property:
\[\sum_{i=1}^{n} T_k(\S)_{ij,i} = 0 \quad \text{for each} \;\; j = 1, \dots, n. \]
Here $ T_k(\S)_{ij,l}$ denotes the components of $\D^M(T_k(\S))$ in a local orthonormal frame on $M$.
\end{lemma}

For each $k = 0, \dots, n$,  the \textit{normalized affine $k$-curvature} is defined by $\widetilde{H}_k \coloneqq \frac{\sigma_k(\S)}{{n \choose k}}$. Now we prove the Minkowski-type formula for $\widetilde{H}_k$.

\begin{theorem}
Suppose that $M$ is closed. Let $x$ be the position vector. We have
\[\int_M \langle \xi, \nu \rangle \widetilde{H}_k = -\int_M \langle x, \nu \rangle \widetilde{H}_{k+1} \quad \text{for each} \;\; k = 0, \dots, n-1.\]
\end{theorem}

\begin{proof}
For a local orthonormal frame $\{e_i\}$ on $M$, by Lemma \ref{lem: D^M X^top_xi} and the fact that $\D^M \langle x, \nu \rangle = - \sff_\nu(x^\top)$, we have
\begin{align*}
\langle \D_{e_i} (x^{\top_\xi}), e_j \rangle = \langle \xi, \nu \rangle \delta_{ij} - \langle \sff_\nu(\xi^\top), e_i \rangle \langle x, e_j \rangle + \langle \sff_\nu(x^\top), e_i \rangle \langle \xi, e_j \rangle + \langle x, \nu \rangle \langle \S(e_i), e_j \rangle.
\end{align*}
In addition, since $\sff_\nu T_k(\S)$ is self-adjoint by Lemma \ref{lem: self-adjoint lemma}, we have
\begin{align*}
&T_k(\S)_{ij}\left(\langle \sff_\nu(x^\top), e_i \rangle \langle \xi, e_j \rangle - \langle \sff_\nu(\xi^\top), e_i \rangle \langle x, e_j \rangle \right)\\
&\qquad = \langle \sff_\nu T_k(\S)(\xi^\top), x^\top \rangle - \langle \sff_\nu T_k(\S)(x^\top), \xi^\top \rangle \\
&\qquad = 0.
\end{align*}
Putting these facts together, we apply Lemma \ref{lem: T_k is divergence-free} to obtain
\begin{align*}
\div_M(T_k(\S)(x^{\top_\xi})) &= T_k(\S)_{ij} \langle \D_{e_i} (x^{\top_\xi}), e_j \rangle \\
&= \langle \xi, \nu \rangle \tr T_k(\S) + \langle x,\nu \rangle \tr(T_k(\S) \cdot \S)\\
&= (n-k) \langle \xi, \nu \rangle \sigma_k(\S) + (k+1) \langle x, \nu \rangle \sigma_{k+1}(\S),
\end{align*}
where we have used \eqref{eq: tr [T_k(A)]A} and \eqref{eq: tr T_k} to derive the last equality. Integrating this over $M$ yields the desired formula.
\end{proof}

\section*{Acknowledgements}	
The author was supported by National Key R\&D Program of China 2020YFA0712800.

\end{document}